\documentclass[12pt,a4paper]{amsart}
\usepackage{graphicx}
\usepackage{color}
\usepackage[colorlinks=true]{hyperref}
\hypersetup{urlcolor=blue, citecolor=blue, draft=false}
\usepackage{subfigure}

\newtheorem{theorem}{Theorem}[section]
\newtheorem{lemma}{Lemma}[section]
\newtheorem{corollary}{Corollary}[section]
\theoremstyle{remark}\newtheorem{example}{Example}[section]

\theoremstyle{remark}\newtheorem{remark}{Remark}[section]

\begin{document}
\title[On a subclass of close-to-convex harmonic mappings]{on a subclass of close-to-convex harmonic mappings}

\author[Manivannan Mathi and Jugal Kishore Prajapat]{Manivannan Mathi and Jugal Kishore Prajapat}

\address{Manivannan Mathi, Department of Mathematics, Central University of Rajasthan, Bandarsindri, Kishangarh-305817, Dist.-Ajmer, Rajasthan, India.}
\email{manivannan.mathi91@gmail.com, jkprajapat@gmail.com}

\begin{abstract}
 For $\alpha > -1$ and $\beta >0, $ let $\mathcal{B}_{\mathcal{H}}^0(\alpha, \beta)$ denote the class of sense preserving harmonic mappings $f=h+\overline{g}$ in the open unit disk $\mathbb{D}$ satisfying $|zh''(z)+\alpha(h'(z)-1)|\leq \beta-|zg''(z)+\alpha g'(z)|.$ First, we establish that each function belonging to this class is close-to-convex in the open unit disk if $\beta \in (0, 1+\alpha]$. Next, we obtain coefficient bounds, growth estimates and convolution properties. We end the paper with applications and will construct harmonic univalent polynomials belonging to this class.
\end{abstract}

\subjclass[2010]{30C45, 30C55, 31A05}
\keywords{Harmonic mappings, Close-to-convex functions, Subordinations, Coefficient estimates, Convolution.}

\maketitle
\section{Introduction}
\setcounter{equation}{0}

        A complex valued mapping  $f=u+iv$ in domain  $\Omega$ is a planar  harmonic mapping,  if both $u$ and $v$ are real-valued  harmonic functions in  $\Omega$. If the  domain $\Omega$  is simply connected and $z_0\in \Omega,$ then the $f$ admits a unique canonical representation $f=h+\overline{g},$ where both $h$ and $g$ are analytic in $\Omega$ and $g(z_0)=0.$ The harmonic mapping  $f$ is locally univalent in $\Omega$ if and only if its Jacobian  $J_f(z) = |f_z(z)|^2-|f_{\overline{z}}(z)|^2$  is non-zero in $\Omega$ (see \cite{lewy}). It is sense preserving, if $J_f(z)>0 \;(z \in \Omega),$ or equivalently, if $h'(z) \neq 0$ and the dilatation $w=g' /h'$ is analytic and satisfies $|w|<1$ in $\Omega.$
        
        We denote by  $\mathcal{H},$ the class of all harmonic  mappings $f=h+\overline{g}$ in the  open unit disk $\mathbb{D}=\{z \in\mathbb{C}: |z|<1\}$ that are normalized by $h(0)=g(0)=h'(0)-1=0$. Each function $f$ in $\mathcal{H}$ can be expressed as $f=h+\overline{g},$ where $h$ and $g$ are analytic functions in $\mathbb{D},$ and have power series representations 
        \begin{equation}\label{intro1}
        h(z) = z + \sum_{n=2}^{\infty} a_nz^n  \quad {\rm and } \quad g(z)=\sum_{n=1}^{\infty} b_nz^n.
        \end{equation}
        Let $\mathcal{S}_\mathcal{H}$ be the subclass of $\mathcal{H}$ consisting  of univalent and sense-preserving \linebreak harmonic mappings in $\mathbb{D}.$ Also, we denote by  $\mathcal{H}^0=\left\lbrace  f\in \mathcal{H}: f_{\overline{z}}(0)=0 \right\rbrace $ and  \linebreak $\mathcal{S}_\mathcal{H}^0=\left\lbrace  f\in \mathcal{S} _\mathcal{H}: f_{\overline{z}}(0)=0 \right\rbrace.$ In 1984, Clunie and Sheil-Small \cite{clunie} investigated the class  $\mathcal{S}_\mathcal{H}$ together with some  of its geometric subclasses. For recent results in harmonic  mappings, we refer to \cite{bshouty1, graf, kanas, zliu, prajap, prajapat} and the references therein.
        
        In \cite{herandez}, Hern\'andez and Martin introduced the concept of stable harmonic mappings. A sense preserving harmonic mapping $f=h+\overline{g}$ is said to be stable harmonic univalent (resp. stable harmonic convex, stable harmonic starlike, or stable harmonic close-to-convex) in $\mathbb{D}$, if all functions $F_\lambda=h+\lambda g$ with $|\lambda|=1$  are univalent (resp. convex, starlike, or close-to-convex) in $\mathbb{D}.$ They proved that for $|\lambda|=1,$ the functions $f_{\lambda}=h+\lambda \overline{g}$ are univalent (resp. convex, starlike, or close-to-convex) for all such $\lambda$ \,(see \cite{herandez}). We recall that, a function $f\in\mathcal{H}$  is said to be close-to-convex, if $f(\mathbb{D})$ is a close-to-convexity, {i.e.,} the complement of $f(\mathbb{D})$ can be written as disjoint union of non-intersecting half lines. The following sufficient condition for close-to-convexity of harmonic mappings is due to Clunie and Sheil-Small \cite{clunie}. 
        
        \begin{lemma}\label{lm.6}
        	If harmonic mapping $f=h+\overline{g}:\mathbb{D}\rightarrow\mathbb{C}$ satisfies $|g'(0)|<|h'(0)|$ and the function  $F_\lambda=h+\lambda g$ is close-to-convex for every $|\lambda|=1,$ then $f$ is close-to-convex function.
        \end{lemma}
        
        An analytic function $\varphi$ is said to  subordinate to the analytic function $\psi$ and written by $\varphi(z)\prec \psi(z)$, if there exists a function $w$ analytic in $\mathbb{D}$ with $w(0)=0$, and $|w(z)|<1$ for all $z\in\mathbb{D}$, such that $\varphi(z)=\psi(w(z))$, $z\in\mathbb{D}$. Furthermore, if the function $\psi$ is univalent in $\mathbb{D}$, then we have the following equivalence:
        $$ \varphi(z)\prec \psi(z)\Longleftrightarrow \;[\varphi(0)=\psi(0)\quad\text{and}\quad \varphi(\mathbb{D})\subset \psi(\mathbb{D})].$$
        In this article, we shall use the following known result of subordination.
        
        \begin{lemma} (see \cite[Ponnusamy Eq. 16]{samy3}). \label{lm.4}  
        	Let $\mathcal{P}$ be an analytic function such that $\mathcal{P}(0)=1.$ Then for real $ \alpha$ such that $\alpha > -1,$ we have 
        	\[ \mathcal{P}(z)+\alpha z \mathcal{P}'(z) \prec 1+\lambda z \Rightarrow  \mathcal{P}(z) \prec  1+\frac{\lambda}{\alpha+1}z, \qquad \qquad  z\in\mathbb{D}. \]
        \end{lemma}     
        
        For two analytic functions  $\psi_1$ and $\psi_2$ in $\mathbb{D},$  given by $\psi_1(z)=\sum_{n=0}^{\infty}a_nz^n$ and $\psi_2(z)=\sum_{n=0}^{\infty}b_nz^n,$ the {\it  convolution (or Hardamard product)} is defined by   $ \left( \psi_1\ast \psi_2 \right) (z)= \sum_{n=0}^{\infty} a_nb_nz^n,$ $z\in\mathbb{D}. $  Analogously, for harmonic functions $f_1=h_1+\overline{g_1}$ and $f_2=h_2+\overline{g_2}$ in $\mathcal{H},$ the  convolution of $f_1$ and $f_2$ is defined as $\; f_1 \ast f_2 = h_1\ast h_2+\overline{g_1\ast g_2}.$\; Clunie and Sheil-Small \cite{clunie} proved  that, if $f$ is harmonic convex function, and $\phi$ is an analytic convex function, then   $f\ast\left(\phi+\alpha \overline{\phi} \right) $ 
        is a harmonic close-to-convex function for all $\alpha$ such that $|\alpha|<1.$ Clearly the space $\mathcal{H}$ is closed under the convolution, {\it i.e.} $\mathcal{H}\ast \mathcal{H}\subset \mathcal{H}.$  We refer \cite{droff1, kumar, li,zliu1} for more information concerning convolution of harmonic mappings and in the case of analytic functions, we refer for examples \cite{samy41,st} and the references therein.
        
        Let $\mathcal{A}$ denote the class of analytic functions in the unit disk $\mathbb{D}$ and are normalized by $f(0)=f'(0)-1=0.$ We also denote by $\mathcal{S},$ the subclass  of $\mathcal{A}$ consisting of univalent functions. Let $\mathcal{S}^*$ and $\mathcal{K}$ denote the subclasses of $\mathcal{A}$, which consists the starlike and convex functions, respectively. A function $f \in \mathcal{A}$ is close-to-convex in $\mathbb{D},$ if there exists a convex analytic function $\phi$ in $\mathbb{D},$ not necessarily normalized, such that $\Re \left(f'(z) /\phi'(z)\right)>0 $ in $\mathbb{D}$. Ponnusamy and Singh  \cite{samy1} have studied a subclass $\mathcal{B}(\alpha,\beta)$ of close-to-convex functions $f\in \mathcal{A}$ which satisfy the condition 
        \[ |zf''(z)+\alpha(f'(z)-1)|<\beta, \qquad z\in\mathbb{D},\]
        where $\alpha>-1$ and $\beta> 0.$ Also, they  proved that  functions  in the  class $\mathcal{B}(\alpha, \beta)$ are convex in $\mathbb{D},$  if $\alpha>-1$ and
        \begin{equation}\label{intro.3}
        0<\beta \leq \left\{\begin{array}{ll}
        \dfrac{1-\alpha}{2+\alpha}  \qquad  {\rm for }  \qquad  & -1< \alpha \leq \sqrt{5}-2,\\
        \dfrac{1+\alpha}{\sqrt{5}}  \qquad  {\rm for }  \qquad  &\sqrt{5} -2\leq \alpha \leq 1,\\
        \dfrac{1+\alpha}{\alpha\sqrt{5}}  \qquad  {\rm for } \qquad  & 1\leq \alpha \leq \dfrac{2}{\sqrt{5}-1},\\
        \dfrac{1+\alpha}{2+\alpha}  \qquad  {\rm for }  \qquad  &\dfrac{2}{\sqrt{5}-1}\leq \alpha \leq 2,\\
        \dfrac{1+\alpha}{2\alpha}  \qquad   {\rm for }  \qquad  &  \alpha \geq 2,   \end{array} \right. 
        \end{equation}
        (see \cite[Corollary 4]{samy1}); and stralike in $\mathbb{D},$ if $\alpha>-1$ and
        \begin{equation}\label{intro.4}
        0<\beta \leq \left\{\begin{array}{ll}
        \dfrac{2(1+\alpha)}{2+\alpha^{2/(1-\alpha)}}  \qquad  {\rm for } \qquad  & -1< \alpha \neq 1 < \infty, \\
        \dfrac{4e^2}{1+e^2} \quad \quad  \qquad  \;\;\; {\rm for }  \qquad  &  \alpha =1,\\
        \end{array} \right.
        \end{equation}
        (see \cite[Theorem 1.14]{samy1}). Further, we deduce the conditions for the univalency of functions in the class $\mathcal{B}(\alpha, \beta)$ by taking    $p(z)=f'(z)-1, \;k=1/\alpha \;(\alpha > -1)$ and $J =\beta/ |\alpha|$ in \cite[Ponnusamy Eq. 16]{samy3}. This provides, if $f \in \mathcal{A}$ and $|zf''(z)+\alpha (f'(z)-1)|< \beta \;(z \in \mathbb{D}),$ then $|f'(z)-1| <\beta /(1+\alpha) \;(z \in \mathbb{D}). $ Therefore, the functions in the $\mathcal{B}(\alpha, \beta)$  are close-to-convex (hence univalent) in $\mathbb{D}$, if $\alpha >-1$  and  $\beta \in (0, 1+\alpha].$
        
        \medskip
        Now we define harmonic analogue of the class $\mathcal{B}(\alpha, \beta)$. For $\alpha > -1$ and $\beta >0,$ let $\mathcal{B}_{\mathcal{H}}^0(\alpha, \beta)$ be a subclass of $\mathcal{H}^0$ which is defined  by
        \begin{align*}
        	& \mathcal{B}_{\mathcal{H}}^0(\alpha, \beta) \\
        	&= \left\lbrace f=h+\overline{g}\in \mathcal{H}^0:  |zh''(z)+\alpha(h'(z)-1)|\leq \beta -|zg''(z)+\alpha g'(z)|, \;z\in \mathbb{D}\right\rbrace.
        \end{align*}
        
        Note that, for $\alpha=0,$ the class $\mathcal{B}_{\mathcal{H}}^0(\alpha, \beta)$  reduces to the class $\mathcal{B}_{\mathcal{H}}^0(\beta),$ which was studied recently by Ghosh and Vasudevaro \cite{ghosh}. Further $\mathcal{B}_{\mathcal{H}}^0(\alpha, \beta)$ reduces to $\mathcal{B}(\alpha, \beta)$, if the co-analytic part of $f$ in $\mathcal{B}_{\mathcal{H}}^0(\alpha, \beta)$ is zero.
        
        In this article, we prove that the functions in $\mathcal{B}_{\mathcal{H}}^0(\alpha, \beta)$  are close-to-convex in $\mathbb{D}.$ We also prove that functions in $\mathcal{B}_{\mathcal{H}}^0(\alpha, \beta)$ are stable harmonic univalent, stable harmonic convex and stable harmonic starlike in $\mathbb{D}$ for different values of its parameters. Further, the coefficient estimates, growth results, area theorem, boundary behaviour, convolution and convex combination properties of the class $\mathcal{B}_{\mathcal{H}}^0(\alpha, \beta)$  of harmonic mapping are obtained. Finally, we consider the harmonic mappings which involve  hypergeometric functions and obtain conditions on its parameters such that it belongs to the class $\mathcal{B}_{\mathcal{H}}^0(\alpha, \beta)$.
         \section{Main Results}
 The first  result provides a one-to-one correspondence between the classes  $\mathcal{B}_{\mathcal{H}}^0(\alpha, \beta)$ and $\mathcal{B} (\alpha, \beta).$
 
 \begin{theorem} \label{2.1}
 	For $\alpha > -1$ and   $\beta> 0$, the harmonic mapping $f=h+\overline{g} \in \mathcal{B}_{\mathcal{H}}^0(\alpha, \beta)$ if and only if $F_\lambda=h+\lambda g \in \mathcal{B}(\alpha, \beta)$ for all $\lambda (|\lambda|=1|).$
 \end{theorem}
 
 \begin{proof}
 	
 	We follow the method of proof of \cite{yama}, for example, let $f=h+\overline{g} \in \mathcal{B}_{\mathcal{H}}^0(\alpha, \beta).$ Then for all $\lambda (|\lambda|=1|),$  we have 
 	\begin{align*}
 		\left| z  F_\lambda''(z)+\alpha\left(  F_\lambda'(z)-1 \right) \right|&= \left|z (h+\lambda g)''(z)+\alpha((h+\lambda g)'(z)-1) \right| \\
 		&\leq \left|z h''(z)+\alpha(h'(z)-1)\right|+ \left| zg''(z)+\alpha g'(z) \right| \\
 	&\leq  \beta,
 	\end{align*}
 	and hence $F_\lambda \in \mathcal{B}(\alpha, \beta).$ Conversely, for all $\lambda (|\lambda|=1)$, let $F_\lambda=h+\lambda g \in \mathcal{B}(\alpha, \beta)$. Then  
 	\begin{align*} 
 		\left| z  F_\lambda''(z)+\alpha\left(F_\lambda'(z)-1 \right) \right| & = \left|z h''(z)+\alpha(h'(z)-1) + \lambda (zg''(z)+\alpha g'(z)) \right|   \\
 		\notag  &< \beta, \qquad z \in \mathbb{D}.
 	\end{align*}  
 	For an appropriate choice of $\lambda,$ we obtain from the last inequality
 	\[\left|z h''(z)+\alpha(h'(z)-1) \right| + \left| zg''(z)+\alpha g'(z) \right| < \beta, \qquad  z\in \mathbb{D},  \]       
 	and hence $f\in \mathcal{B}_{\mathcal{H}}^0(\alpha, \beta).$
 \end{proof}    

 \begin{remark} \label{rm1}
 	We observe that functions in $\mathcal{B}_{\mathcal{H}}^0(\alpha, \beta)$ are stable harmonic close-to-convex if $\alpha>-1$ and $\beta \in (0, 1+\alpha]$, stable harmonic convex in $\mathbb{D}$ if $\alpha >-1$ and $\beta$ satisfies the conditions (\ref{intro.3}), stable harmonic starlike in $\mathbb{D}$ if $\alpha >-1$ and $\beta$ satisfies the conditions (\ref{intro.4}). Further, the following result shows that functions in $\mathcal{B}_{\mathcal{H}}^0(\alpha, \beta)$.
 \end{remark}
 
 \begin{theorem} \label{newthm}
 	For $\alpha >-1$ and $\beta  \in (0, 1+\alpha],$ the harmonic mappings in $\mathcal{B}_{\mathcal{H}}^0(\alpha, \beta)$ are close-to-convex in $\mathbb{D}.$
 \end{theorem}
 \begin{proof}
 	If $ f \in \mathcal{B}_{\mathcal{H}}^0(\alpha, \beta)$, then by theorem \ref{2.1}, we have  $F_\lambda=h+\lambda g \in \mathcal{B}(\alpha, \beta)$ for all $\lambda(|\lambda|=1).$ Hence, $F_\lambda$ are close-to-convex in $\mathbb{D}$ for $\alpha >-1$ and $\beta  \in (0, 1+\alpha].$ Now using Lemma \ref{lm.6}, we conclude that functions in $\mathcal{B}_{\mathcal{H}}^0(\alpha, \beta)$ are close-to-convex in $\mathbb{D}.$
 \end{proof}
 
 \begin{theorem} \label{2.3}
 	Let $\alpha >-1$ and $\beta  \in (0, 1+\alpha].$ If $f=h+\overline{g} \in \mathcal{B}_{\mathcal{H}}^0(\alpha, \beta),$ then 
 	\begin{equation*} 
 	|z|-\frac{\beta}{2(1+\alpha)}|z|^2 \leq|f(z)|\leq |z|+ \frac{\beta}{2(1+\alpha)}|z|^2, \qquad z\in\mathbb{D}.
 	\end{equation*}
 	Both the inequalities are sharp for the functions 
 	$$f_1(z)=z+\dfrac{\beta}{2(1+\alpha)}z^2 \quad \mbox{and} \quad f_2(z)=z+\dfrac{\beta}{2(1+\alpha)}\overline{z}^2,$$
 	and their rotations.   
 	
 	\begin{proof} If $ f \in \mathcal{B}_{\mathcal{H}}^0(\alpha, \beta)$, then $F_\lambda=h+\lambda g \in \mathcal{B}(\alpha, \beta)$ for all $\lambda(|\lambda|=1).$  Hence
 		$$zF''_{\lambda}(z)+\alpha F'_{\lambda}(z) \prec \alpha +\beta z, \quad z \in \mathbb{D}.$$
 		Using Lemma \ref{lm.4}, we obtain 
 		$$F'_{\lambda}(z) \prec 1+\frac{\beta}{1+\alpha} z, \quad z \in \mathbb{D}.$$
 		Therefore
 		\begin{equation*}
 		1-\frac{\beta}{1+\alpha}|z|\leq \left |F_\lambda'(z) \right |=\left | h'(z)+\lambda g'(z)\right |\leq 1+\frac{\beta}{1+\alpha}|z|.
 		\end{equation*}
 		Since $\lambda (|\lambda|=1) $ is arbitrary, it follows that
 		\begin{equation*}
 		|h'(z)| +|g'(z)| \leq 1+ \frac{\beta}{1+\alpha}|z|
 		\end{equation*}
 		and
 		\begin{equation*}
 		|h'(z)| -|g'(z)| \geq 1-\frac{\beta}{1+\alpha}|z|.
 		\end{equation*}    
 		If $ \Gamma$ is the radial segment from 0 to $z,$ then 
 		\[ |f(z)| = \left|   \int_{\Gamma}  \frac{\partial f}{\partial \xi}d \xi   + \frac{\partial f}{\partial \overline{\xi}} d\overline{\xi}\right | \leq \int_{\Gamma} \left( |h'(\xi)|+|g'(\xi)| \right) |d \xi |  \] 
 		\[\quad \leq \int_{0}^{|z|} \left( 1+\frac{\beta}{1+\alpha}t \right)dt= |z|+\frac{\beta}{2(1+\alpha)}|z|^2,  \]
 		and 
 		\[ |f(z)| = \left|   \int_{\Gamma}  \frac{\partial f}{\partial \xi}d \xi   + \frac{\partial f}{\partial \overline{\xi}} d\overline{\xi}\right | \geq \int_{\Gamma} \left( |h'(\xi)|-|g'(\xi)| \right) |d \xi |  \qquad \] 
 		\[ \geq \int_{0}^{|z|} \left( 1-\frac{\beta}{1+\alpha}t \right)dt= |z|-\frac{\beta}{2(1+\alpha)}|z|^2,  \]
 		which completes the proof of the  theorem.
 	\end{proof}
 \end{theorem}     
 
 The following theorem provides sharp coefficient bounds for functions in  $\mathcal{B}_{\mathcal{H}}^0(\alpha, \beta).$   
 
 \begin{theorem}\label{2.2}
 	Let $f=h+\overline{g}\in \mathcal{B}_{\mathcal{H}}^0(\alpha, \beta)$ be given by (\ref{intro1}), then for $n\geq 2,$
 	\[  \left|a_n\right| \leq \dfrac{\beta}{n\left( n+\alpha-1\right) }  \qquad  {\rm and  } \qquad      \left|b_n\right| \leq \dfrac{\beta}{n\left( n+\alpha-1\right) }.\]   
 	Both the inequalities  are sharp.  
 \end{theorem}
 \begin{proof}
 	The proof follows from the method of \cite[Proof of theorem 2]{li1}, but for the sake of completeness we include it here. Let $f =h+\overline{g} \in \mathcal{B}_{\mathcal{H}}^0(\alpha, \beta)$ be given by (\ref{intro1}), then $|zh''(z)+\alpha(h'(z)-1)|<\beta.$ Now using Cauchy's theorem, we have  
 	$$n\left( n+\alpha-1\right)a_n=\frac{1}{2\pi i } \int_{|z|=r} \frac{zh''(z)+\alpha(h'(z)-1)}{z^n} dz,  \qquad |z|=r<1,$$ 
 	and hence the bound for $|a_n|$ follows. Also, using the trivial bound $|a_n|+|b_n| \leq \dfrac{\beta}{n(n+\alpha-1)}$  $(n \geq 2)$, the bounds for $|b_n|$ follows. 
 	The sharpness of result can be shown by taking
 	$$ f_1(z)= z+\frac{\beta}{n\left( n+\alpha-1\right) }z^n \quad  {\rm and  } \quad \it {f_2(z)}= \it {z}+\frac{\it{\beta}}{\it{n}\left( n+\alpha-1\right) } \overline{z}^n.$$
 	This completes the proof of the theorem.
 \end{proof}

 The following result gives a sufficient condition for functions belonging  to the  class $ \mathcal{B}_{\mathcal{H}}^0(\alpha, \beta).$ 
 
 \begin{theorem}\label{2.4}
 	Let $\alpha >-1$ and $\beta >0.$  If $f=h+\overline{g} \in \mathcal{H}^0$ be given by (\ref{intro1}) and  
 	\begin{equation} \label{2.4.1}
 	\sum_{n=2}^{\infty} n (n+\alpha-1)(|a_n|+|b_n|)\leq \beta,
 	\end{equation}
 	then  $f \in \mathcal{B}_{\mathcal{H}}^0(\alpha, \beta).$ 
 \end{theorem}
 \begin{proof}
 	Let $f=h+\overline{g}  \in \mathcal{H}^0.$  Then from (\ref{intro1}) and (\ref{2.4.1}), we obtain
 	\begin{align*}
 		\left |zh''(z) +\alpha \left(h'(z)-1 \right) \right | &= \left | \sum_{n=2}^{\infty} n(n-1)a_nz^{n-1}+\alpha \sum_{n=2}^{\infty}n a_nz^{n-1}\right |\\
 		&\leq   \sum_{n=2}^{\infty}n(n+\alpha-1)|a_n||z|^{n-1} \\
 		& \leq \beta -\sum_{n=2}^{\infty}n(n+\alpha-1)|b_n|\\
 		&\leq  \beta - \left | \sum_{n=2}^{\infty}n(n+\alpha-1)b_nz^{n-1} \right |\\
 		&= \beta-\left | zg''(z)+\alpha g'(z) \right|, 
 	\end{align*}
 	and hence $f\in \mathcal{B}_{\mathcal{H}}^0(\alpha, \beta).$
 \end{proof}

 In view of Remark \ref{rm1} and Theorem \ref{2.4}, we obtain following interesting corollaries. 
 
 \begin{corollary} \label{c1} 
 	Let $\alpha >-1$ and $\beta \in (0, 1+\alpha].$ If $f=h+\overline{g}\in \mathcal{H}^0$ is given by (\ref{intro1}) and satisfies the inequality (\ref{2.4.1}), then $f$ is stable harmonic close-to-convex in $\mathbb{D}.$
 \end{corollary}
 
 \begin{corollary} \label{c2}
 	Let $ \alpha>-1$ and $\beta$ satisfy the condition (\ref{intro.3}). If $f=h+\overline{g}\in \mathcal{H}^0$ is given by (\ref{intro1}) and satisfies the inequality (\ref{2.4.1}), then $f$ is stable harmonic convex in $\mathbb{D}.$  
 \end{corollary} 
 
 \begin{corollary} \label{c3}
 	Let $ \alpha>-1$ and $\beta$ satisfy the condition (\ref{intro.4}).  If $f=h+\overline{g}\in \mathcal{H}^0$ is given by (\ref{intro1}) and satisfies the inequality (\ref{2.4.1}), then $f$ is stable harmonic starlike in $\mathbb{D}.$  
 \end{corollary} 
 
 \begin{example}
 	Consider the function
 	\begin{equation}\label{example}
 	\theta_{\alpha, \beta}(z)= z+\frac{\beta}{4(1+\alpha)}(z^2-\overline{z}^2) \qquad (\alpha>-1, \beta >0, z \in \mathbb{D}).
 	\end{equation}
 	Then we have
 	$$\sum_{n=2}^{\infty}n(n+\alpha-1)\left( |a_n|+|b_n|\right) =\beta.$$
 	Hence $\theta_{\alpha, \beta} \in \mathcal{B}_{\mathcal{H}}^0(\alpha, \beta).$ Also, $\theta_{\alpha, \beta}$ 	is stable harmonic close-to-convex in $\mathbb{D}$  \linebreak if $\alpha >-1, \beta \in (0, 1+\alpha]$, stable harmonic convex in $\mathbb{D}$ if $ \alpha>-1$ and $\beta$ satisfy the condition (\ref{intro.3}), stable harmonic starlike if $ \alpha>-1$ and $\beta$ satisfy the condition (\ref{intro.4}).
 \end{example}
 
 The following results shows that the boundary of $f(\mathbb{D})$ is a rectifiable Jordan curve for each $ f\in \mathcal{B}_{\mathcal{H}}^0(\alpha, \beta).$
 
 \begin{theorem}\label{2.6}
 	For real $\alpha$ and  $\beta$ such that  $\alpha>-1$ and $\beta \in (0, 1+\alpha],$ each function in $ \mathcal{B}_{\mathcal{H}}^0(\alpha, \beta)$ maps the $\mathbb{D}$ onto a domain which is bounded by a rectifiable Jordan curve.
 \end{theorem}
 
 \begin{proof}
 	Each function  $ f=h+\overline{g}\in\mathcal{B}_{\mathcal{H}}^0(\alpha, \beta)$ is uniformly continuous in  $\mathbb{D},$ and hence can be extended continuously onto $|z|=1.$ To see this, let $z_1$ and $z_2$ be  two distinct points in $\mathbb{D},$ and  $\left[ z_1,z_2  \right] $ be the line segment joining $z_1$ to $z_2.$ We have from  (\ref{2.4.1})  
 	\begin{align*}\label{2.6.1} 
 		|f(z_1)-f(z_2)| &=\left|   \int_{[z_1,z_2]}  \frac{\partial f}{\partial \xi}d \xi   + \frac{\partial f}{\partial \overline{\xi}} d\overline{\xi}\right | \\ 
 		&\leq  \int_{|z_2|}^{|z_1|} \left( 1+\frac{\beta}{1+\alpha} t \right) dt  \\  
 		&= (|z_1|-|z_2|)\left(1 +\frac{\beta}{2(1+\alpha)}\left( |z_1|+|z_2| \right) \right) \\
 		& \leq  2 (|z_1|-|z_2|)\leq 2|z_1-z_2|. 
 	\end{align*}
 	Now, let the curve $\mathcal{C}$  be defined by $w=f(e^{i \theta}),$ $0\leq \theta \leq 2\pi. $ If $0=\theta_0<\theta_1<\cdots \theta_n=2 \pi $ is a partition of $[0,2\pi],$ then  using (\ref{2.6.3}), we obtain 
 	\[\sum_{k=1}^{n} \left|   f(e^{i\theta_k})-f(e^{i\theta_{k-1}} )\right| \leq 2\sum_{k=1}^{n} \left|  e^{i\theta_k}-e^{i\theta_{k-1}}  \right| < 4\pi ,  \]
 	which shows that $\mathcal{C}$ is a rectifiable curve. It remains to show that $f$ is univalent on $\partial \mathbb{D}= \left\lbrace  z \in \mathbb{C}: |z|=1 \right\rbrace. $ In view of Theorem \ref{2.1},  functions $F_\lambda=h+\lambda g$ belongs to the class $\mathcal{B}(\alpha, \beta)$ 
 	for all $|\lambda|=1.$ In particular each $F_\lambda$ is univalent in $\partial \mathbb{D}$ by \cite[Theorem 3]{mac}.  
 	
 	Now suppose that $z_1,z_2$ are two distinct points on $\partial \mathbb{D}= \left\lbrace  z: |z|=1\right\rbrace$ such that \linebreak $f(z_1)=f(z_2).$ Then
 	\begin{equation} \label{2.6.3}
 	\ h(z_1)-h(z_2) = \overline{g(z_2)-g(z_1).}
 	\end{equation}
 	If $h(z_1)=h(z_2),$ then $g(z_1)=g(z_2)$ and so $z_1=z_2$ by the univalence of $F_1$. Now assume that $h(z_1) \neq h(z_2),$ and let  $\theta =arg \left\lbrace h(z_1)-h(z_2) \right\rbrace \in [0,2\pi). $ 
 	Then $ e^{-i \theta} \left( h(z_1)-h(z_2)\right)$ is a  positive real number, now multiplying  (\ref{2.6.3}) by $e^{-i\theta}$ and  taking the conjugate on both sides,  we have 
 	\[h(z_1)-h(z_2)=e^{2i\theta} \left( g(z_2)-g(z_1)\right),   \]
 	which implies that $F_{\lambda}(z_1)=F_{\lambda}(z_2)$ with $\lambda=e^{2i\theta}.$ Thus $z_1=z_2,$ which shows that $f$ is univalent in $\partial \mathbb{D}.$  This  completes the proof of Theorem \ref{2.6}.
 \end{proof}

 Now, we will show that the class $\mathcal{B}_{\mathcal{H}}^0(\alpha, \beta)$ is closed under convex combinations. Also, we  show that for $\phi \in \mathcal{K}$ and $ f\in \mathcal{B}_{\mathcal{H}}^0(\alpha, \beta),$ the function $f\ast\left(\phi+\beta \overline{\phi} \right)\in \mathcal{B}_{\mathcal{H}}^0(\alpha, \beta) $ for all $|\beta|=1.$  To show that the class $\mathcal{B}_{\mathcal{H}}^0(\alpha, \beta)$ is closed under convex combinations, we shall need following Lemma:
 
 \begin{lemma} [see \cite{singhr}] \label{lm.5} Let $p$ be an analytic function in $\mathbb{D},$ with $p(0)=1$ and \linebreak $\Re (p(z))>1/2$  in $\mathbb{D}.$ Then for any analytic function $f$ in $\mathbb{D},$ the function $p\ast f$ takes  values in the convex hull of the image of $\mathbb{D}$ under $f.$  
 \end{lemma}
 
 \begin{theorem} \label{3.1}
 	The class $\mathcal{B}_{\mathcal{H}}^0(\alpha, \beta)$ is closed under convex combination.
 \end{theorem} 
 \begin{proof}
 	Let $f_k=h_k+\overline{g_k}\in\mathcal{B}_{\mathcal{H}}^0(\alpha, \beta)$ for $k=1,2,\cdots  n,$ and $\sum_{k=1}^{n}t_k=1 \linebreak (0\leq t_k\leq 1).$
 	The convex combination of the $f_k$'s can be written as  
 	\[ f(z)=\sum_{k=1}^{n} t_kf_k(z)=h(z)+\overline{g(z)},\]
 	where $h(z)=\sum_{k=1}^{n} t_kh_k(z)$  and $g(z)=\sum_{k=1}^{n} t_kg_k(z).$   A computation shows that 
 	\begin{align*}
 		\left | zh''(z)+ \alpha(h'(z)-1)\right| &= \left |\sum_{k=1}^{n} t_k ( z h''_k(z) +\alpha (h'_k(z) -1)) \right |\\
 		& \leq   \sum_{k=1}^{n} t_k \left | z h''_k(z) +\alpha (h'_k(z) -1) \right | \\
 		&= \sum_{n=1}^{n}t_k \left(  \beta- |zg_k''(z)+\alpha g'_k(z)|\right)\\ 
 		&\leq  \beta - \left|z   \sum_{k=1}^{n}t_kg_k''(z)+\alpha\sum_{k=1}^{n}t_kg_k'(z)\right | \\
 		&= \beta- |zg''(z)+\alpha g'(z)|, 
 	\end{align*} 
 	and so $f\in \mathcal{B}_{\mathcal{H}}^0(\alpha, \beta). $
 \end{proof}

 \begin{theorem} \label{3.2}
 	Let $f\in \mathcal{B}_{\mathcal{H}}^0(\alpha, \beta) $ and $\phi \in \mathcal{K}.$ Then  $f\ast\left(\phi+\lambda \overline{\phi} \right)\in \mathcal{B}_{\mathcal{H}}^0(\alpha, \beta) $ for all $\lambda \;(|\lambda|=1).$ 
 \end{theorem}
 
 \begin{proof}
 	Let $f=h+\overline{g}$ be in $\mathcal{B}_{\mathcal{H}}^0(\alpha, \beta). $ Then 
 	\[ f\ast\left(\phi+\lambda \overline{\phi} \right) = h\ast \phi +\overline{\overline{\lambda}(g\ast \phi)}. \]
 	It sufficient to show that $F_{\lambda}= h\ast \phi +\overline{\lambda} \left( g\ast \phi\right)\in\mathcal{B}(\alpha, \beta)$ for all $\lambda(|\lambda|=1).$ A computation shows that 
 	\begin{equation} \label{3.2.1}
 	zF''_{\epsilon}(z)+\alpha\left( F'_\epsilon (z)-1\right)=\left( z\left(h+\lambda g \right)''(z) +\alpha \left( (h+\lambda g)'-1 \right)  \right)\ast \dfrac{\phi(z)}{z}.  
 	\end{equation}
 	Since $f=h+\overline{g} \in \mathcal{B}_{\mathcal{H}}^0(\alpha, \beta),$ the function $h+\overline{\lambda} g \in \mathcal{B}(\alpha, \beta),$ and so  
 	\[\left |  zh''(z) +\alpha\left( h'(z)-1 \right)+ \overline{\lambda}(zg''(z)+ \alpha g'(z) ) \right | \leq \beta, \qquad z\in\mathbb{D}. \]
 	Since $\phi \in \mathcal{K},$  implies that  $\Re\left(\dfrac{\phi(z)}{z} \right)>\dfrac{1}{2} $ in $\mathbb{D}.$ Now applying Lemma \ref{lm.5}, we obtain that
 	\[ \left |  zF_\lambda ''(z)+\alpha (F_\lambda'(z)-1)\right | \leq \beta, \qquad z\in\mathbb{D}. \]
 	Hence $F_\lambda \in \mathcal{B}(\alpha, \beta)$ for all $\lambda \;(|\lambda|=1)$, equivalently $f\ast\left(\phi+\lambda \overline{\phi} \right)\in \mathcal{B}_{\mathcal{H}}^0(\alpha, \beta)$ for all  $\lambda \;(|\lambda|=1)$.
 \end{proof}

 \section{Applications}
 
 In this section, we consider  harmonic mappings which involve the  Gaussian \linebreak hypergeometric function and obtain conditions so that such harmonic mappings belongs to the class $\mathcal{B}_{\mathcal{H}}^0(\alpha, \beta)$. The Gaussian hypergeometric function $_2F_1(a,b;c;z)$ is defined by 
 \begin{equation} \label{4.1}
 F (a,b;c;z)=\, _2F_1(a,b;c;z)=\sum_{n=0}^{\infty} \frac{(a)_n (b)_n}{(c)_n n!}z^n,  
 \end{equation}
 where $a, b, c \in \mathbb{C}, \,c \neq 0, -1, -2, -3, \cdots$ and $(x)_n$ is the Pochhammer symbol defined by $(x)_0=1, \;(x)_{n+1}=(x+n)(x)_n=x(x+1)_n \;(n=0, 1, 2, \cdots)$. The series (\ref{4.1}) is absolutely convergent in $\mathbb{D}.$ Moreover, if $\Re(c-a-b) >0$, then series (\ref{4.1}) is convergent in $|z|\leq 1.$ The well-known Gauss formula (see \cite{nm})) for hypergeometric function is given by $_2F_1(a,b;c;1)=\Lambda $ for $ \Re(c-a-b)>0,$ and where
 \[\Lambda=\frac{ \Gamma(c)\Gamma(c-a-b)}{\Gamma(c-a)\Gamma(c-b)}. \]
 We shall use the following Lemma to prove result in this section. 
 
 \begin{lemma}[see \cite{samy4}] \; \label{lm4.1} Let $a,b \in \mathbb{R}\setminus\{0\}$ and $c$ is a positive real number.  Then the following holds
 	\begin{itemize}
 		\item[(a)] For $  c>a+b+1,$
 		\[\sum_{n=0}^{\infty}\dfrac{(n+1)(a)_n (b)_n}{(c)_nn! } =  \dfrac{\Gamma(c)\Gamma(c-a-b-1)}{\Gamma(c-a)\Gamma(c-b)}(ab+c-a-b-1).\]
 		
 		\item[(b)] For $ c>a+b+2,$
 		\[\sum_{n=0}^{\infty}\dfrac{(n+1)^2(a)_n (b)_n}{(c)_n n! } =  \dfrac{ \Gamma(c)\Gamma(c-a-b)}{\Gamma(c-a)\Gamma(c-b)} \left( \dfrac{(a)_2(b)_2}{(c-a-b-2)_2} +\dfrac{3ab}{c-a-b-1}+1\right).\]
 		
 		\item[(c)] For  $a \neq 1, b \neq 1$ and $c \neq 1$ with $ c>\mbox{max}\{0, a+b+1\},$
 		\[\sum_{n=0}^{\infty}\dfrac{(a)_n (b)_n}{(c)_n (n+1)!} = \frac{1}{(a-1)(b-1)} \left(\frac{\Gamma(c)\Gamma(c-a-b-1)}{\Gamma(c-a)\Gamma(c-b)}-(c-1) \right).\]
 	\end{itemize}
 \end{lemma} 
 
 \medskip
 Below we use the ideas used by \cite{bharanedhar,samy31} for the univalency of harmonic mappings involving the Gaussian hypergeometric functions.
 The first result in this section is given by
 
 \begin{theorem} \label{thm4.1} Let $a,b \in \mathbb{R}\setminus\{0\}$ and $c$ is a positive real number. Suppose that $\ f_1(z)=z+\overline{ z^2 F(a,b;c;z)}, f_2(z)=z+\overline{z \left( F(a,b;c;z)-1\right)}$ and $f_3(z)=z+\overline{z \int_{0}^{z} F(a,b;c;t)dt}$, then the following holds 
 	
 	\medskip
 	\noindent
 	(a) If $c>a+b+2$ and
 	\begin{align} \label{4.3}
 	\frac{(a)_2 (b)_2}{(c-a-b-2)_2}+ \frac{ab(\alpha+4)}{c-a-b-1}+2(1+\alpha) \leq \frac{\beta}{\Lambda}, 
 	\end{align}
 	then $f_1 \in \mathcal{B}_{\mathcal{H}}^0(\alpha, \beta). $
 	
 	\medskip
 	\noindent
 	(b) If $c>a+b+2$ and
 	\begin{equation} \label{4.4}
 	\frac{ab(ab+c-1)}{(c-a-b-2)_2}+\frac{ab(1+\alpha)}{c-a-b-1}+\alpha \leq \frac{\beta-\alpha}{\Lambda}, 
 	\end{equation}
 	then $f_2 \in \mathcal{B}_{\mathcal{H}}^0(\alpha, \beta). $
 	
 	\medskip
 	\noindent
 	(c)  If  $a \neq 1, b \neq 1$ and $c \neq 1$ with $ c>\mbox{max}\{0, a+b+1\}$ and
 	\begin{equation} \label{4.5}
 	\Lambda \left(\frac{ab}{c-a-b-1}+\frac{\alpha}{(a-1)(b-1)(c-a-b-1)}+\alpha \right)- \frac{\alpha(c-1)}{(a-1)(b-1)}\leq \beta,
 	\end{equation}
 	then $f_3 \in \mathcal{B}_{\mathcal{H}}^0(\alpha, \beta). $  
 \end{theorem}
 
 \begin{proof}
 	(a) Let $f_1(z)=z+ \overline{\sum_{n=2}^{\infty} C_nz^n},$ where $C_n=   \frac{(a)_{n-2} (b)_{n-2}}{(c)_{n-2}(n-2)!} \;\; (n\geq 2).$ Using Lemma \ref{lm4.1} and Gauss formula, we have      
 	\begin{align} \label{4.6}
 	\sum_{n=2}^{\infty}n(n+\alpha-1)|C_n|&=  \sum_{n=2}^{\infty} n(n+\alpha-1)\frac{(a)_{n-2} (b)_{n-2}}{(c)_{n-2}(n-2)!} \notag  \\
 	& = \sum_{n=0}^{\infty} (n+1)^2 \frac{(a)_n (b)_n}{(c)_nn! }  +(1+\alpha) \sum_{n=0}^{\infty} (n+1)\frac{(a)_n (b)_n}{(c)_nn! } \\
 	&+ \alpha \sum_{n=0}^{\infty} \frac{(a)_n (b)_n}{(c)_n \,n!}\notag \\
 	& =   \Lambda \left( \frac{(a)_2 (b)_2}{(c-a-b-2)_2}+ \frac{ab(\alpha+4)}{c-a-b-1}+2(1+\alpha) \right) . \notag 
 	\end{align}
 	Now if (\ref{4.3}) holds, then $\sum_{n=2}^{\infty}n(n+\alpha-1)|C_n|\leq \beta.$ Now using Theorem \ref{2.4}, we  conclude that $f_1 \in \mathcal{B}_{\mathcal{H}}^0(\alpha, \beta). $
 	
 	\medskip
 	\noindent  
 	(b) Let  $\ f_2(z)=z+ \overline{\sum_{n=2}^{\infty} D_nz^n},$ where $D_n =  \dfrac{(a)_{n-1} (b)_{n-1}}{(c)_{n-1}(n-1)!} \;\; (n\geq 2).
 	$ Using Lemma \ref{lm4.1} and Gauss formula, we have 
 	\begin{align} \label{4.7}
 	\sum_{n=2}^{\infty}n(n+\alpha-1)|D_n| &= \sum_{n=2}^{\infty} n(n+\alpha-1)\frac{(a)_{n-1} (b)_{n-1}}{(c)_{n-1}(n-1)!} \notag\\ 
 	& = \sum_{n=0}^{\infty}  (n+1) \frac{(a)_{n+1} (b)_{n+1}}{(c)_{n+1} \,n!} + (1+\alpha) \sum_{n=0}^{\infty}\frac{(a)_{n+1} (b)_{n+1}}{(c)_{n+1} \,n!} \notag \\
 	& \qquad +\alpha \sum_{n=0}^{\infty} \frac{(a)_{n+1} (b)_{n+1}}{(c)_{n+1}(n+1)!} \notag \\
 	&=\Lambda \left[\frac{ab(ab+c-1)}{(c-a-b-2)_2}+\frac{ab(1+\alpha)}{c-a-b-1}+\alpha \right] -\alpha. \notag
 	\end{align}   
 	Now if (\ref{4.4}) holds, then in view of Theorem \ref{2.4}, we have $f_2 \in \mathcal{B}_{\mathcal{H}}^0(\alpha, \beta).$
 	
 	\medskip
 	\noindent   
 	(c) Let $f_3(z)=z+ \overline{ \sum_{n=2}^{\infty}E_nz^n},$  where $E_n =  \dfrac{(a)_{n-2} (b)_{n-2}}{(c)_{n-2}(n-1)!} \;\; n\geq 2.$ Therefore in view of Lemma \ref{lm4.1} and Gauss formula, we have
 	\begin{align*}
 		\sum_{n=2}^{\infty} n(n+\alpha-1)|E_n| &= \sum_{n=2}^{\infty} n(n+\alpha-1)\frac{(a)_{n-2} (b)_{n-2}}{(c)_{n-2}(n-1)!} \\ 
 		&=\sum_{n=0}^{\infty} (n+1) \frac{(a)_{n} (b)_{n}}{(c)_{n} \,n!} +(1+\alpha)\sum_{n=0}^{\infty} \frac{(a)_{n} (b)_{n}}{(c)_{n} \,n!} \\
 		& +\alpha \sum_{n=0}^{\infty} \frac{(a)_{n} (b)_{n}}{(c)_{n} \,(n+1)!} \notag
 	\end{align*}
 	\begin{align*}
 		&= \frac{\Gamma(c) \Gamma(c-a-b-1)}{\Gamma(c-a) \Gamma(c-b)}(ab+c-a-b-1)+(1+\alpha)\frac{\Gamma(c) \Gamma(c-a-b)}{\Gamma(c-a) \Gamma(c-b)} \notag \\
 		& \qquad +\frac{\alpha}{(a-1)(b-1)}\left(\frac{\Gamma(c) \Gamma(c-a-b-1)}{\Gamma(c-a) \Gamma(c-b)} -(c-1) \right). 
 	\end{align*}
 	If (\ref{4.5}) holds, then by Theorem \ref{2.4}, we have  $f_3 \in \mathcal{B}_{\mathcal{H}}^0(\alpha, \beta). $                                   
 \end{proof}

 \medskip
 Note that for $ \eta \in\mathbb{C} \setminus \left\lbrace -1,-2, \cdots \right\rbrace $  and $n\in\mathbb{N} \cup \left\lbrace 0 \right\rbrace, $ we have 
 \[\frac{(-1)^n(-\eta)_n}{n!}=\binom{\eta }{n}=\frac{\Gamma 
 	(\eta+1)}{n!\Gamma (\eta-n+1)}.\] 
 In particular, when $\eta=m \left( m\in\mathbb{N}, m\geq n \right),$ we have
 \[ (-m)_n =\frac{(-1)^nm!}{(m-n)!}.\]
 Using this relation in Theorem \ref{thm4.1}, we can obtain harmonic univalent \linebreak polynomials that belong to the class $\mathcal{B}_{\mathcal{H}}^0(\alpha, \beta).$ \\
 
 \begin{corollary} Let $m\in \mathbb{N},c$ be a positive real numbers.  Let 
 	$$F_1(z)=z+\overline{\sum_{n=0}^{m}\binom{m}{n} \frac{(m-n+1)_n}{(c)_n}z^{n+2}}, \quad F_2(z)=z+\overline{\sum_{n=0}^{m}\binom{m}{n} \frac{(m-n+1)_n}{(c)_n}z^{n+1}}$$
 	and 
 	$$F_3(z)=z+\overline{  \sum_{n=0}^{m}\binom{m}{n}\frac{(m-n+1)_n}{(c)_n} \frac{z^{n+2}}{n+1}}. \qquad \qquad \qquad \qquad \qquad \qquad \qquad \qquad$$
 	Then the following holds.
 	
 	\medskip
 	\noindent
 	(a) If 
 	\[\frac{m^2(m-1)^2}{(c+2m-1)(c+2m-2)}+\frac{m^2(\alpha+4)}{c+2m-1}+2(1-\alpha) \leq \frac{ \beta \left[ \Gamma (c+m)\right]^2  }{\Gamma (c) \Gamma (c+2m)},\]
 	then $F_1 \in \mathcal{B}_{\mathcal{H}}^0(\alpha, \beta).$
 	
 	\medskip
 	\noindent
 	(b) If
 	\[\frac{m^2(c+m^2-1)}{(c+2m-2)(c+2m-1)}+\frac{m^2(1+\alpha)}{c+2m-1}+\alpha \leq \frac{(\beta-\alpha) \left[ \Gamma (c+m)\right]^2  }{\Gamma (c) \Gamma (c+2m)},\]
 	then $F_2 \in \mathcal{B}_{\mathcal{H}}^0(\alpha, \beta). $
 	
 	\medskip
 	\noindent
 	(b) If \[\frac{\Gamma(c)\Gamma(c+2m) }{\left( \Gamma(c+2m)\right]^2 }\left[ \frac{m^2}{c+2m-1}+\frac{\alpha}{(m+1)^2(c+2m-1)}+\alpha \right)-\frac{\alpha(c-1)}{(m+1)^2} \leq \beta, \]
 	then $F_3 \in \mathcal{B}_{\mathcal{H}}^0(\alpha, \beta).$  
 \end{corollary}  
 
 \begin{proof}
 	The results follow, if we put $a=b=-m$ in Theorem 3.1.
 \end{proof}

 \medskip
 We conclude this paper by remarking that, by appropriately selecting parameters in Theorem 3.1 and Corollary 3.1, our results would lead to new results and further \linebreak applications. These consideration can fruitfully be worked out and we skip the details in this regards. 
 
 \medskip
 \section{Acknowledgements}
 The authors express their sincere thanks to the editor and referees for their valuable suggestions to improve the manuscript.

\end{document}